\documentclass{amsart}
\usepackage{graphicx}
\usepackage{amssymb,amscd,amsthm,amsxtra}
\usepackage{latexsym}
\usepackage{epsfig,esint}

\newtheorem{thm}{Theorem}[section]
\newtheorem{cor}[thm]{Corollary}
\newtheorem{lem}[thm]{Lemma}
\newtheorem{prop}[thm]{Proposition}
\theoremstyle{definition}
\newtheorem{defn}[thm]{Definition}
\theoremstyle{remark}
\newtheorem{rem}[thm]{Remark}
\numberwithin{equation}{section}

\newcommand{\R}{\mathbb R}

\newcommand{\C}{\mathcal C}

\newcommand{\be}{\begin{equation}}
\newcommand{\ee}{\end{equation}}

\newcommand{\eps}{\epsilon}
\newcommand{\ov}{\overline}
\newcommand{\p}{\partial}

\newcommand{\comment}[1]{}

\begin{document}

\title{Obstacle type problems for minimal surfaces}

\author{L. Caffarelli}
\address{Department of Mathematics, University of Texas at Austin, Austin, TX 78712, USA}\email{\tt caffarel@math.utexas.edu}
\author{D. De Silva}
\address{Department of Mathematics, Barnard College, Columbia University, New York, NY 10027, USA}
\email{\tt  desilva@math.columbia.edu}
\author{O. Savin}
\address{Department of Mathematics, Columbia University, New York, NY 10027, USA}\email{\tt  savin@math.columbia.edu}
\thanks{D.~D. is supported by NSF grant DMS-1301535.  O.~S.~is supported by  NSF grant DMS-1200701.}

\begin{abstract} We study certain obstacle type problems involving standard and nonlocal minimal surfaces. We obtain optimal regularity of the solution 
and a characterization of the free boundary. \end{abstract}

\maketitle
\section{Introduction and main results}

In this note, we investigate the regularity of the solution and of the free boundary, for certain obstacle type problems involving classical minimal surfaces and nonlocal minimal surfaces first introduced in \cite{CRS}.

Our first main results concerns  the optimal regularity for nonlocal minimal surfaces constrained below by a sufficiently smooth obstacle (see Section 2 for the precise definitions). 

\begin{thm}\label{main_reg_OMS_intro} Let $\mathcal O \subset \R^n$ be a $C^{1,\alpha}$ domain 
(obstacle), with $\alpha > s+\frac 12$, $s \in (0, \frac 1 2)$. Assume that $E$ is fixed outside $B_1$ and that it minimizes the $s$-perimeter in $B_1$ among all sets that contain $\mathcal O \cap B_1$. If
$0 \in \p E \cap \p \mathcal O$, then $\p E \cap B_{\delta_0}$ is a $C^{1,\frac 12 +s}$ surface in 
$B_{\delta_0}$, for some $\delta_0$ depending on $n,s$ and the $C^{1,\alpha}$ norm of $\p \mathcal O$.
\end{thm}

The key ingredients in the proof of Theorem \ref{main_reg_OMS_intro} are an improvement of flatness lemma and a suitable version of the Almgren monotonicity formula. The strategy is to write $\p E$ as the graph of a 
function $u$ of $n-1$ variables, and reduce the problem to an obstacle problem for 
$\triangle^{\frac 12 + s}u$.

 Once this result is established, we consider the two membranes problem between a standard minimal 
surface and a nonlocal minimal surface. 

The two membranes problem refers to the equilibrium position of two elastic membranes constrained one on top of the other. In the set where they do not touch, the membranes will satisfy a prescribed PDE. The two membranes problem for the Laplacian was first considered by Vergara-Caffarelli \cite{VC} in the context of variational inequalities. See for example \cite{ARS,CDS,CCV,CV} for further results. In particular, in \cite{CDS} we considered the more challenging case when the two membranes satisfy different linear PDEs. Here we consider a nonlinear two membranes problem for different operators. Precisely, we study the two membranes problem between a standard minimal 
surface and a nonlocal minimal surface. The problem can be formulated as follows.

For $s \in (0, \frac 12)$ and $\Omega$ a bounded set in $\R^n,$ let 
\be\label{ps}\mathcal P^s_{\Omega}(E) := L(E \cap \Omega, \C E) + L(E\setminus \Omega, \C E \cap \Omega)\ee be
the fractional $s$-perimeter of the set $E \subset \R^n$ in $\Omega$, introduced in \cite{CRS}, where
$L(A,B)$ represents 
\be\label{L}
L(A,B): = \int_{\R^n \times \R^n} \frac{1}{|x-y|^{n+2s}} \chi_A(x)\chi_B(y) dx dy,\ee
and $\C A$ denotes the complement of the set $A$. 
Let us also denote by 
$$Per_{\Omega}(F) = \int_{\Omega}|\nabla \chi_F| ,$$
the perimeter of a set $F \subset \R^n$ in $\Omega$.

 Define the functional,
$$
\mathcal G(E,F) : =\mathcal P^s_{B_1}(E) + Per_{B_1}(F) + \int_{B_1} (f \chi_E  + g \chi_F) dx $$
with $f, g \in L^\infty(B_1).$
Let $E_0$, $F_0$ be two sets such that $F_0 \cap B_1 \subset E_0 \cap B_1$ and $\mathcal G
(E_0,F_0) < \infty$.  
We define the class of admissible pairs of sets in $\R^n$:
$$\mathcal A:= \{(E,F) \ | F \subset E \quad \mbox{in $B_1$}, \quad F=F_0, \quad E=E_0 \quad 
\mbox{outside $B_1$}  \}.$$

We minimize $\mathcal G$ over the class $\mathcal A$ and study the regularity properties of the minimizing pair $(E,F)$. Notice that the two surfaces may interact, that is $\p E \cap \p F \neq \emptyset,$ independently of the sign of $f,g.$

We show that $\p F$ is an {\it almost 
minimal surface} in the sense of Almgren-Tamanini \cite{A,T}. This means that $\p F$ is a $C^{1,\alpha}$ surface except 
on a singular set $\Sigma \subset \p F$ of Hausdorff dimension $n-8$. Using our main Theorem \ref{main_reg_OMS_intro}, we also obtain the optimal 
regularity of the minimizing pair $(E,F)$ away from the singular set $\Sigma$ of $\p F$.

\begin{thm}\label{T4}
Assume that $f,g \in C^{\frac 12 -s}(B_1)$. Then $\p F \setminus \Sigma$ is locally a 
$C^{2,\frac 12 -s}$ surface and $\p E$ is locally a $C^{1,\frac 12+s}$ surface 
in a neighborhood of $\p F \setminus \Sigma$.  
\end{thm}

This theorem is optimal even when $f$ and $g$ vanish. It says that generically, the two 
membranes $\p E$ and $\p F$ separate at different rates away from their common part $\p E \cap 
\p F$. In general they are not smooth across the free boundary as it can be seen from a simple 1D 
example.

In view of Theorem \ref{T4} above, in a neighborhood of a point of $\p E \cap \p F \setminus \Sigma,$ the problem 
can be reduced to the two membranes problem for the fractional Laplacian. Therefore using Theorem 2.6 in \cite{CDS} we obtain 
a characterization of the free boundary of the coincidence set around so-called ``regular" points. We describe it below.

Assume $0 \in (\p E \cap \p F) \setminus \Sigma$ and that in a neighborhood of $0$, $$\p E=\{(x', u(x'))\}, \p F = \{(x', v(x'))\},$$
with $u \in C^{1, \frac 1 2 +s}, v \in C^{2, \frac 1 2 -s}.$ Let $Q$ be the coincidence set $$Q:= \{x' \ | \ u(x')=v(x')\}.$$ 
We say that $x_ 0 \in \p Q$ is a regular point if
$$\limsup_{r \to 0} r^{-\frac 3 2 -s} \|u-v\|_{L^\infty(B_r(x_0))} > 0.$$

\begin{cor} If  $x_0 \in \p Q$ is a regular point, then $\p Q$ is an $n-2$ dimensional $C^{1,\gamma}$ surface around $x_0$. 
\end{cor}

The paper is organized as follows. In Section 2 we consider the obstacle problem for nonlocal minimal surfaces and prove our main Theorem \ref{main_reg_OMS_intro}. In doing so, we also show $C^{1,1}$ regularity of ``flat" nonlocal minimal surfaces and provide a sharp quantitative estimate of the norm, depending on the flatness. This improves the result in \cite{CRS}. Higher regularity estimates for nonlocal minimal surfaces were also proved in \cite{BFV}. We prove Theorem \ref{T4} in Section 3. 

\section{The obstacle problem for nonlocal minimal surfaces}

In this section we prove the optimal regularity for nonlocal minimal surfaces constrained above a sufficiently smooth obstacle, that is our main Theorem \ref{main_reg_OMS_intro}. First, we recall some definitions. 

For $s \in (0, \frac 12)$ and $\Omega$ a bounded set in $\R^n,$
let $\mathcal P^s_{\Omega}(E) $ be
the fractional $s$-perimeter of the set $E \subset \R^n$ in $\Omega$, as defined in \eqref{ps}.

\begin{defn} We say that $\p E$ is a $s$-minimal surface in $\Omega$ if for any set $F$ with $F \cap (\C \Omega) = E \cap (\C \Omega)$ we have
$$\mathcal P^s_\Omega (E) \leq \mathcal P^s_\Omega(F).$$ 
\end{defn}


\smallskip

Given a set $E$, we identify $E$ with its interior points in a measure theoretical sense, hence $\p E$ is a closed set. We say that a point $x \in \p E$, is regular from above (resp. below) if there exists a tangent ball to $x$ at $\p E$ completely contained in $\C E$ (resp. $E$).
 
\begin{defn} We say that $\p E$ is a viscosity $s$-minimal supersolution in $\Omega$ if at any  $x_0 \in \p E \cap \Omega$ regular point from below, 
$$K_E(x_0):= \int_{\R^n} \frac{\chi_E - \chi_{\C E}}{|x-x_0|^{n+2s}} dx \leq 0.$$
\end{defn}

Analogously, one defines viscosity $s$-minimal subsolutions. If $\p E$ is both a viscosity $s$-minimal subsolution and supersolution, then we say that $\p E$ is a viscosity $s$-minimal surface. We remark that the quantity $K_E$ is well defined at all regular points and represents the ``fractional curvature" of $\p E.$


\smallskip

We recall our main Theorem \ref{main_reg_OMS_intro}. 

\begin{thm}\label{main_reg_OMS} Let $\mathcal O \subset \R^n$ be a $C^{1,\alpha}$ domain 
(obstacle), with $\alpha > s+\frac 12$. Assume that $E$ is fixed outside $B_1$ and that it minimizes the $s$-perimeter in $B_1$ among all sets that contain $\mathcal O \cap B_1$. If
$0 \in \p E \cap \p \mathcal O$, then $\p E \cap B_{\delta_0}$ is a $C^{1,\frac 12 +s}$ surface in 
$B_{\delta_0}$, for some $\delta_0$ depending on $n,s$ and the $C^{1,\alpha}$ norm of $\p \mathcal O$.
\end{thm}

\begin{rem}We remark that if $\alpha< s+ \frac 12$ then our methods give that $\p E$ is as regular as the obstacle, i.e. $\p E\in C^{1,\alpha}$ in a neighborhood of the origin.\end{rem}

Clearly $\p E$ is a nonlocal s-minimal surface in $\C \ov{\mathcal O} \cap B_1$. 
In fact,  (see Theorem 5.1 in \cite{CRS})
\begin{equation}\label{vsob1}
\mbox{$\p E$ is a viscosity $s$-minimal surface in $B_1 \setminus \ov{\mathcal O}$, and}
\end{equation}
\be \label{vsob2}
\mbox{$\p E$ is a viscosity supersolution in $B_1$.}
\ee

The theorem above deals with the regularity of the constrained minimal surface at the points where $\p E$ sticks to the obstacle $\p \mathcal O$. In the lemma below, we observe that around such points $E$
 satisfies a flatness condition. Precisely, the following holds.
 
 \begin{lem}\label{f_red}Let $\mathcal O \subset \R^n$ be a $C^{1}$ obstacle. Assume that $E$ is fixed outside $B_1$ and that it minimizes the $s$-perimeter in $B_1$ among all sets that contain $\mathcal O \cap B_1$. If
$0 \in \p E \cap \p \mathcal O$, and $x_n=0$ is the tangent plane to $\p \mathcal O$ at 0, then for any $\eps>0$ there exists $r(\eps)$ such that
\be\label{de}
\{x_n <-\eps r\} \subset E  \subset \{x_n < \eps r\} \quad \mbox{in $B_{r}$.}\ee
 \end{lem}
 \begin{proof} Given any $\delta>0$, we can assume that (possibly after rescaling),
 $$\{x_n < -\delta\} \cap B_1 \subset \mathcal O \subset \{x_n < \delta\} \cap B_1.$$
Fix $\eps>0$ and $\delta =\frac{\eps r}{2} $ and let us prove that \eqref{de} holds, with $r$ to be specified later. Clearly the left inclusion is satisfied. We prove the other inclusion. Assume it does not hold, and let $x_0 \in B_{r} \cap \p E \cap \{x_n \geq r\eps\}$. Then $B_\delta(x_0)$ is included in $\{x_n > \delta\}$, hence $\p E$ is unconstrained in this ball. By the density estimates (see \cite{CRS}), there is a ball of radius comparable to $\delta$ fully contained in $E.$ 

Assume first for simplicity that $\p E$ has a tangent parabola of unite size by below at $0$. Then, for $C_0,C_1$ universal
$$0 \geq K_E(0) \geq - C_0 + C_1 \frac{\delta^n}{r^{n+2s}}.$$ where the second term comes from the contribution in $B_{\delta}(x_0)$. Hence,
$$\eps \leq C r^{\frac{2s}{n}},$$
and we get a contradiction for $r$ small enough.

If $0$ is not regular from below, then we slide from below a parabola of unit size. Since $0 \in \p E \cap \p \mathcal O,$ the first touching point will occur at $y$ such that $$|y'| \leq \sqrt \delta, \quad |y_n| \leq \delta.$$ Thus the previous argument can be easily repeated with $0$ replaced by $y$, by choosing an appropriate $r$.
 \end{proof}

The first step toward the proof of Theorem \ref{main_reg_OMS} is to show almost optimal regularity near the boundary $\p \mathcal O$ for the non-local minimal surface, i.e. that $\p E$ is as 
regular as the obstacle $\p \mathcal O$ up to $C^{1,\beta}$ with $\beta< \frac 12 +s$. In view of Lemma \ref{f_red}, we can consider the case when $E$ is flat. 

 From now on, a point $x=(x',x_n)$ with $x'=(x_1,\ldots,x_{n-1})$ and $B'_{r}$ denotes the $n-1$ dimensional ball of radius $r$ centered at 0. Also,
$$P_0:=\{x \cdot e_n<0\}.$$ 
and $$\mathcal O := \{x_n < \varphi(x')\}, \quad \varphi \in C^{1,\alpha}.$$

\begin{thm}\label{imp_OMS}[Almost Optimal Regularity] Assume $E$ satisfies \eqref{vsob1}-\eqref{vsob2} and
$$
\{ x_n < -\eps\}  \subset E \subset \{x_n<\eps \}  \quad \mbox{in $B_1$},$$
$$ \int_{\R^n \setminus B_1} \frac{|\chi_E - \chi_{P_0}|}{|x|^{n+2s+1}} dx \leq \eps,$$ and
$$[\nabla \varphi]_{C^\alpha(B'_1)} \leq \eps.$$
Then, for all $\beta$ with $\beta < s+\frac 12$ and $\beta \le \alpha$, there exist $\eps_0, C$ depending on $\alpha, \beta, s, n,$ such that if $\eps \leq \eps_0$
$$\p E \cap B_{1/2} = \{(x', u(x')) \ | \ x' \in B'_{1/2}\}$$ with 
$$\|u\|_{C^{1,\beta}(B'_{1/2})} \leq C \eps.$$
\end{thm}

We first prove this theorem in the case of unconstrained $s$-minimal surfaces. Indeed, the two proofs are essentially the same and this second theorem is interesting in its own. It improves the $C^{1,\beta}$ regularity, $\beta <2s$, of flat nonlocal minimal surfaces, to all $\beta <1$ (see Theorem 6.1 in \cite{CRS}). It also provide a sharp quantitative estimate of the norm, depending on the flatness.

In what follows, constants depending only on $n,s$ are called universal and may change from line to line. 

\begin{thm}\label{imp_MS} Let $\p E$ be a viscosity $s$-minimal surface in $B_1$. Assume that
\begin{equation}\label{need}\{ x_n < -\eps\}  \subset E \subset \{x_n<\eps \}  \quad \mbox{in $B_1$},\end{equation}
 and 
$$ \int_{\R^n \setminus B_1} \frac{|\chi_{E} - \chi_{P_0}|}{|x|^{n+2s+1}} dx \leq \eps.$$
Then, for all $\beta<1$, there exist $\eps_0, C$ depending on $\beta, s, n$ such that if $\eps \leq \eps_0$
$$\p E \cap B_{1/2} = \{(x', u(x')) \ | \ x' \in B'_{1/2}\}$$ with 
$$\|u\|_{C^{1,\beta}(B'_{1/2})} \leq C \eps.$$
\end{thm}

Theorem \ref{imp_MS} easily follows from the next improvement of flatness lemma.

\begin{lem}\label{flat_imp}Let $\p E$ be a viscosity $s$-minimal surface in $B_1$ satisfying \eqref{need}, with $0\in \p E$, 
\be\label{flat1}
\p E \cap B_1 \subset \{|x_n| \leq \eps\},\ee and
\be \label{flat2} \int_{\R^n \setminus B_1} \frac{|\chi_{E} - \chi_{P_0}|}{|x|^{n+2s+1}} dx \leq \eps.\\
\ee
Then, for any $\beta<1$, there exist $\eps_0$,$\rho$ depending on $\beta, s, n$, such that if $\eps \leq \eps_0$ then
\begin{equation}
\p E \cap B_{\rho} \subset \{|x\cdot e| \leq \eps \rho^{1+\beta}\},
\end{equation}
and 
\be
 \int_{\R^n \setminus B_\rho} \frac{|\chi_{E} - \chi_{P_e}|}{|x|^{n+2s+1}} dx \leq  \eps \rho^{\beta -1-2s},
\ee
 for some unit vector $e$ and $P_e=\{x \cdot e <0\}$.
\end{lem}

Notice that $\frac 1 \rho E$ satisfies the hypotheses above with $\eps$ replaced by $\eps \rho^\beta$. Therefore we can iterate this lemma indefinitely and obtain the desired conclusion in Theorem \ref{imp_MS}.

\begin{proof} 
The proof is similar to the improvement of flatness theorem for nonlocal minimal surfaces from 
\cite{CRS}, except that in this case we work with truncated kernels and this allows us to impose less 
restrictive conditions at infinity.

We divide the proof in three steps that we sketch below. \medskip
 
\textbf{Step 1-- Truncation.}  
Define,
\begin{equation}\label{KET}f_{E}(y): = \int_{\R^n \setminus B_{1/4}(y)}\frac{\chi_{\C E} - \chi_{E}}{|x-y|^{n+2s}}dx, \quad y \in \R^n.\end{equation}
By minimality, if $y \in \p E$ is a regular point, then 
\begin{equation}\label{KET2}f_E(y) =  p.v. \int_{B_{1/4}(y)}\frac{\chi_E - \chi_{\C E}}{|x-y|^{n+2s}}dx.\end{equation}
We claim that  
\be \label{fe0}
|f_E| \leq C \eps, \quad \mbox{ in $B_1 \cap \{|x_n| < \eps \}$,}
\ee  
\be\label{fe1}|f_E(y) - f_E(z)| \leq C( \eps |y'-z'| + |y_n-z_n|) , \quad \quad \forall \, y,z \in B_{5/8} \cap \{|x_n| < \eps\},\ee with $C$ universal.

To prove this we set
$$P_y:= \{(x-y) \cdot e_n < 0\},$$
and by definition we have $f_{P_y}(y)=0$. Thus
\be \label{fe12}f_E(y)= f_E(y)-f_{P_y}(y)=2\int_{\R^n \setminus B_{1/4}(y)}\frac{\chi_{P_y} - \chi_{E}}{|x-y|^{n+2s}}dx,\ee
where we used that $\chi_{\C E}-\chi_E = 1-2 \chi_E $. 

Let $y,z \in B_1 \cap\{|x_n|<\eps \}$ and denote by $d:= |y-z|$, $B(y)=B_{1/4}(y)$,  $B(z)=B_{1/4}(z)$ and $D= B(y) \cup B(z)$.

Using \eqref{fe12} we find
\begin{align*}
\frac 12 |f_E(y) - f_E(z)| \leq & \int_{\C D}|\chi_E - \chi_{P_y}| \left |\frac{1}{|x-y|^{n+2s}}- \frac{1}{|x-z|^{n+2s}}\right | \, dx \\
&+  \int_{\C D} \frac { |\chi_{P_z}-\chi_{P_y}|}{|x-z|^{n+2s}} dx\\
&+  \int_{D \setminus B(y)}\frac{|\chi_E - \chi_{P_y}|}{|x-y|^{n+2s}}dx +\int_{D \setminus B(z)}\frac{|\chi_E - \chi_{P_z}|}{|x-z|^{n+2s}}dx\\
 =:& \, I_1+I_2+I_3 + I_4.
\end{align*}

We estimate,
$$I_1 \le C d \eps,  $$
by using \eqref{flat1}-\eqref{flat2} and that in $\C D$
$$ \left |\frac{1}{|x-y|^{n+2s}}- \frac{1}{|x-z|^{n+2s}}\right | \le \frac{C d}{1+|x|^{n+1+2s}}.$$
Clearly $$I_2 \le C |y_n-z_n|.$$

We estimate $I_3$ and $I_4$ as
$$I_3 \le C \int_{D \setminus B(y)}  |\chi_E - \chi_{P_y}| \, dx,$$
hence by \eqref{flat1}-\eqref{flat2} we have $I_3 \le C \eps$. Moreover, if $|z| \le 5/8$ we use only hypothesis \eqref{flat1} and estimate the measure of $E \Delta P_y$ in $D \setminus B(y)$ and obtain
$$I_3 \le C d \eps.$$

We estimate $I_4$ similarly, and this proves \eqref{fe1}.

From the computations above we see that in order to prove \eqref{fe0} it suffices to find two points $y_0$, $y_1 \in \p E \cap B_1$ such that 
 $$f_E(y_0) \geq - C\eps, \quad f_E(y_1) \leq  C \eps.$$ 
 Indeed, let us slide by below the parabola $x_n = -C \eps |x'|^2$. We touch $\p E$ at a first point $y_0 \in \p E$. Denote by $P$ the subgraph of the tangent parabola to $\p E$ at $y_0$. Since $y_0$ is a regular point by below we write (see \eqref{KET2})
 \begin{equation}\label{fe}f_E(y_0)\geq \int_{B_{1/4}(y_0)} \frac{\chi_P-\chi_{\C P}}{|x-y_0|^{n+2s}} dx   \geq - C \int_0^{1/4} \frac{\eps r^n}{r^{n+2s}} dr \geq - C \eps.\end{equation}

Similarly, sliding a parabola by above, we obtain the point $y_1.$

 \medskip
 
\textbf{Step 2 -- Harnack Inequality.} In this step we show that there exists a universal $\delta$ such that either
$$\p E \cap B_\delta \subset \{\frac{x_n}{\eps} \leq 1-\delta^2\}$$
or
$$\p E \cap B_\delta \subset \{\frac{x_n}{\eps} \geq 1+ \delta^2\}.$$
Moreover, this statement can be iterated a number of times that tends to $\infty$ as $\eps \to 0$.
We argue similarly as in Lemma 6.9 in \cite{CRS}, but using \eqref{fe0} to control the nonlocal contribution. 

We know that $E$ contains $\{x_n < -\eps\} \cap B_1.$ Assume that it contains more than half the measure of the cylinder 
$$D:=\{|x'| \leq \delta\} \times \{|x_n| \leq \eps\}.$$
Then we show that $E$ must contain $$\{x_n \geq (-1+\delta^2)\eps\} \cap B_\delta.$$
Indeed, if the conclusion does not hold, then we slide by below the parabola $$x_n= -\frac{\eps}{2}|x'|^2.$$
The first touching point $y\in \p E$ satisfies
$$|y'|\leq 2\delta, \quad |y_n+\eps| \leq 2 \eps \delta^2.$$
Let $P$ be the subgraph of the tangent parabola to $\p E$ at $y$. Then,
\begin{align}\label{estimate1}
 f_E(y)  & = \int_{B_{1/4}(y)} \frac{\chi_E - \chi_{{\mathcal C E}}}{|x-y|^{n+2s}} dx \\
\nonumber & = \int_{B_{1/4}(y)}    \frac{\chi_P - \chi_{{\mathcal C P}}}{|x-y|^{n+2s}} dx + \int_{B_{1/4}(y)}    \frac{\chi_{E \setminus P}}{|x-y|^{n+2s}} dx\\
\nonumber &=:I_1+I_2.\\ \nonumber
\end{align}
For $\eps \leq \delta$ we estimate
$$I_1 \geq -C \int_0^{1/4} \frac{\eps r^n}{r^{n+2s}} dr \geq - C \eps,$$
and since $E \setminus P$ contains more than 1/4 of the measure of $D$,
$$I_2 \geq C\eps \delta^{n-1}/(4 \delta)^{n+2s} \geq C \delta^{-1-2s} \eps.$$
 If $\delta$ is smaller than a universal constant, we contradict \eqref{fe0}.

The fact that the lemma can be iterated follows because, after rescaling, the H\"older modulus of continuity of $\p E$ outside $B_1$ is integrable at $\infty$ and it does not affect the computations above.  Indeed, assume that we can iterate our Harnack inequality $k$ times and let us call $$\tilde E = \delta^{-k} E, \quad \tilde \eps = \eps (1-\frac{\delta^2}{2})^k \delta^{-k}.$$ Then, for $m=0,\ldots, k$ \be \label{iteration2} \p \tilde E \cap B_{\delta^{-m}} \subset S_m, \ee where $S_m$ is a strip of height $2\tilde \eps (1-\frac{\delta^2}{2})^{-k} \delta^{k-m}$ and say for simplicity that $0 \in S_m.$
To iterate one more time, we need to estimate $f_{\tilde E}$ as in \eqref{estimate1}. Clearly, for $\tilde y= \delta^{-k} y, y \in \p E$,
$$K_{\tilde E, \frac{\delta^{-k}}{4}}(\tilde y):= \int_{\R^n \setminus B_{\frac{\delta^{-k}}{4}}(y)}\frac{\chi_{\C \tilde E} - \chi_{\tilde E}}{|x-\tilde y|^{n+2s}}dx = \delta^{2sk} f_{E}(y),$$
and hence $$|K_{\tilde E, \frac{\delta^{-k}}{4}}| \leq C \tilde \eps.$$
 Using this fact and \eqref{iteration2} we can bound  $f_{\tilde E}$ as desired.


\medskip

\textbf{Step 3 -- Compactness.} Fix $\rho>0$ to be specified later. Assume by contradiction that there exist a sequence $\eps_k \to 0$ and a sequence $\{\p E_k\}$ of viscosity $s$-minimal surfaces in $B_1$ with $0\in \p E_k$, satisfying
\be\label{flatk1}
\p E_k \cap B_1 \subset \{|x_n| \leq \eps_k\},\ee
\be\label{flatk2} \int_{\R^n \setminus B_1} \frac{|\chi_{E} - \chi_{P_0}|}{|x|^{n+2s+1}} dx \leq \eps_k,
\ee but not the conclusion of the lemma. Then, by Step 2, the sets
$$\p E^*_k := \{(x', \frac{x_n}{\eps_k}) \ | \ x \in \p E_k\}$$
converge uniformly on $B'_{3/4} \times [-1,1]$ (up to extracting a subsequence) to the set 
$$E^*_0:= \{(x', u(x'))\}$$
where 
$$u: B'_{3/4} \to \R \quad \mbox{ is H\"older continuous}, \quad \mbox{$u(0)=0$ and $|u| \leq 1$.}$$ Moreover, by \eqref{fe0}-\eqref{fe1} the functions $$f_{E_k^*}(x):= \frac{1}{\eps_k}f_{E_k}(x', \eps_k x_n)$$ are uniformly Lipschitz continuos, and
(up to extracting a subsequence) $$f_{E^*_k} \to f_0$$ uniformly on the cylinder $B'_{3/4} \times [-1,1]$ with $f_0$  Lipschitz with norm controlled by a universal constant.

Next we show that $u$ satisfies 
\begin{equation}\label{DT}
\Delta_T^{\frac 1 2+s} u = \frac 12 f_0(x',u) \quad \text{in $B'_{1/2}$},
\end{equation}
in the viscosity sense, were $\Delta_T^{\frac12 +s}$ is the truncated fractional Laplacian
$$\Delta_T^{\frac 12+s} u(x') = \int_{B'_{1/4}(x')} \frac{u(y') - u(x')}{|y'-x'|^{n+2s}} dy'.$$

Let $\varphi$ be a smooth function which touches $u$ strictly by below, say for simplicity at 0. Then, for $k$ large enough, $\p E^*_k$ is touched by below at some $x^*_k$, by a vertical translation of $\varphi$. 
Hence $\p E_k$ is touched by below at $x_k=({x^*_k}', \eps_k  \, \, x^*_k \cdot e_n)$ by a vertical translation of $\eps_k \varphi$, and 
$$f_{E_k^*}(x_k^*)=\frac{1}{\eps_k} f_{E_k}(x_k) =  \frac{1}{\eps_k}\int_{B_{1/4}(x_k)} \frac{\chi_{E_k} - \chi_{{\mathcal{C}E_k}}}{|x-x_k|^{n+2s}} dx .$$
We first remark that we can change the domain of integration from balls to cylinders since this creates only a small error. Indeed, denote by
$$D_{\rho}(x_k) := B'_{\rho}(x'_k) \times \{|(x-x_k)\cdot e_n| < \rho\}.$$
Since $\p E_k \cap B_1 \subset \{|x_n| \leq \eps_k\}$ we see that the measure of $D_{1/4}(x_k) \setminus B_{1/4}(x_k)$ that is included in the strip $\{|x_n| \le \eps_k \}$ is of order $\eps_k^3$, hence
\be\label{circle-square} \frac{1}{\eps_k}\int_{B_{1/4}(x_k)} \frac{\chi_{E_k} - \chi_{{\mathcal{C}E_k}}}{|x-x_k|^{n+2s}} dx =  \frac{1}{\eps_k}\int_{D_{1/4}(x_k)} \frac{\chi_{E_k} - \chi_{{\mathcal{C}E_k}}}{|x-x_k|^{n+2s}} dx + o_{\eps_k}(1), \ee
with $o_{\eps_k}(1) \to 0$ as $\eps_k \to 0$.

We bound the right hand side in \eqref{circle-square} similarly as in Lemma 6.11 \cite{CRS}. 

 We integrate first 
in the cylinder $D_{1/4}(x_k) \cap \{|x'-x_k'| \le \delta\}$ and use that the graph of $\eps_k \varphi$ is tangent by below at $x_k$. The contribution in this cylinder is greater than $-o_\delta(1) \to 0$ as $\delta \to 0$. 

In the remaining part of $D_{1/4}(x_k)$ we integrate first in the vertical $x_n$ direction and cancel the parts from $E_k$ and $\C E_k$. We are left with an integration over an oriented segment of length $2 \eps_k(u(x')-u(x'_k) + o_{\eps_k}(1))$ included in the strip $\{ |x_n| \le \eps_k\}$ and on this segment we can write
$$\frac{1}{|x-x_k|^{n+2s}}= \frac{1}{|x'-x_k'|^{n+2s}} + C(\delta) O(\eps_k^2) .$$

In conclusion
\begin{align*}
\frac{1}{\eps_k}\int_{D_{\frac 14}(x_k)} \frac{\chi_{E_k} - \chi_{{\mathcal{C}E_k}}}{|x-x_k|^{n+2s}} dx \ge & 2 \int_{B'_{\frac 1 4}(x'_k) \setminus B'_\delta(x'_k)} \frac{u(x')-u(x_k')}{|x'-x_k'|^{n+2s}} dx'\\ & - o_{\delta}(1) - C(\delta) o(\eps_k) + o_{\eps_k}(1).\end{align*}

We let first $\eps_k \to 0$ and then $\delta \to 0$ and we obtain that \eqref{DT} holds  in the viscosity sense at $0$.
 
Thus our claim \eqref{DT} holds, and by interior estimates for this equation we have that the $C^{1,1}$ norm of $u$ in $B'_{1/4}$ is bounded by a universal constant.
Then, for all $\eta\leq 1/4$, any point $x$ on $\p E_0^* \cap B_\eta$ satisfies
$$|x \cdot \nu| \leq C \eta^2,\quad \nu=\frac{(-\nabla u(0),1)}{\sqrt{|\nabla u(0)|^2+1}}$$
with $C$ universal. Since $\p E^*_k \to \p E_0^*$ uniformly on compacts, we conclude that for $k$ large, 
\be\label{quadratic}\p E_k \cap B_\eta \subset \{|x \cdot \nu_k| \leq C \eps_k \eta^2 + o(\eps_k)\},\quad \nu_k= \frac{(\eps_k \nu', 1)}{\sqrt{1+\eps^2_k|\nu'|}}.\ee
Hence, for $\rho$ small (depending on $\beta$),
$$\p E_k \cap B_\rho \subset \{|x \cdot \nu_k| \leq \eps_k \rho^{1+\beta}\}.$$
Now, call 
$$P_k := \{x \cdot \nu_k < 0\}.$$ Then, 
\begin{align*}
\int_{\C B_{\rho}} \frac{|\chi_{E_k} - \chi_{P_k}|}{|x|^{n+2s+1}} dx & =  \int_{\C B_{1/4}} \frac{|\chi_{E_k} - \chi_{P_k}|}{|x|^{n+2s+1}} dx+  \int_{B_{\frac 14} \setminus B_{\rho}} \frac{|\chi_{E_k} - \chi_{P_k}|}{|x|^{n+2s+1}} dx \\
&=I_1+I_2.
\end{align*}
We show that by choosing $\rho$ possibly smaller, the terms above are bounded by $\eps_k \rho^{\beta-1-2s}$, therefore contradicting that $E_k$ does not satisfy the conclusion of the lemma.

Indeed, since $E_k$ satisfies \eqref{flatk1},\eqref{flatk2} and $|e_n -\nu_k|\leq C \eps_k,$ we have
$$I_1  \leq C \eps_k.$$
Moreover, using \eqref{quadratic}
$$I_2 \leq C \int_{\rho}^\frac 14 \frac{\eps_k r^2}{r^{n+2s+1}}r^{n-2}dr = C\eps_k \rho^{-2s}.$$
\end{proof}

We now prove Theorem \ref{imp_OMS}, that is the case when the nonlocal minimal surface is constrained by an obstacle $\mathcal O$. We need the following improvement of flatness lemma whose proof follows the lines of Lemma \ref{flat_imp}.

\begin{lem}\label{imp_OMS_2} Assume $E$ satisfies \eqref{vsob1}-\eqref{vsob2} with $0 \in \p E$ and 
$$
\p E \cap B_1 \subset \{|x_n| \leq \eps\},$$
\be\label{F2} \int_{\R^n \setminus B_1} \frac{|\chi_E - \chi_{P_0}|}{|x|^{n+2s+1}} dx \leq \eps,\ee and
\be\label{F3} 
[\nabla \varphi]_{C^\alpha(B'_1)} \leq \delta_0 \eps.\ee
If $\beta < s+ \frac 12$ and $\beta \le \alpha$, and $\eps \le \eps_0$ is sufficiently small then
\begin{equation}
\p E \cap B_{\rho} \subset \{|x\cdot e| \leq \eps \rho^{1+\beta}\},
\end{equation}
and 
\be
 \int_{\R^n \setminus B_\rho} \frac{|\chi_{E} - \chi_{P_e}|}{|x|^{n+2s+1}} dx \leq  \eps \rho^{\beta-1-2s},
\ee
 for some unit vector $e$ and $P_e=\{x \cdot e <0\}$.
 The constants $\delta_0$, $\eps_0$ and $\rho$ are small and depend on $\alpha, \beta, s, n$.
 \end{lem}
 
 We remark that we obtain Theorem \ref{imp_OMS} by applying the lemma above with $\tilde \eps=\delta_0^{-1} \eps$ and then iterate it indefinitely.

\begin{proof} In what follow we denote by $\p_f E$ and $\p_c E$ the free part of $\p E$, and respectively the constrained part of $\p E$, i.e. 
$$\p_{f} E := \p E \cap \C \ov{\mathcal O}, \quad \p_c E: = \p E \cap \p \mathcal O.$$
\textbf{Step 1-- Truncation.}  
Define as before,
\begin{equation}\label{KET00}f_E(y) = \int_{\R^n \setminus B_{1/4}(y)}\frac{ \chi_{\C E}-\chi_E }{|x-y|^{n+2s}}dx, \quad y \in \R^n.\end{equation}
By minimality, if $y \in \p_f E$ is a regular point, then 
\begin{equation}\label{KET200}f_E(y) =  p.v. \int_{B_{1/4}(y)}\frac{\chi_E - \chi_{\C E}}{|x-y|^{n+2s}}dx.\end{equation}On the other hand, if $y \in \p  E$ is regular from below with respect to a graph, then
 \begin{equation}\label{KET300}f_E(y) \geq  p.v. \int_{B_{1/4}(y)}\frac{\chi_E - \chi_{\C E}}{|x-y|^{n+2s}}dx.\end{equation}
As in Step 1 in Lemma \ref{flat_imp} we see that $f_E$ satisfies \eqref{fe1}, and due to \eqref{KET300} only half of \eqref{fe0}, i.e.
$$f_E \geq - C \eps, \quad \mbox{ in $B_1 \cap \{|x_n| < \eps \}$.}$$

 \medskip
 
\textbf{Step 2 -- Harnack Inequality.} In this step we show as in Lemma \ref{flat_imp} that there exists a universal $\delta$ such that either
\be\label{1}\p E \cap B_\delta \subset \{\frac{x_n}{\eps} \leq 1-\delta^2\}\ee
or
\be\label{2}\p E \cap B_\delta \subset \{\frac{x_n}{\eps} \geq 1+ \delta^2\}.\ee

Let us slide from above the parabola,
$$x_n= \frac{\eps}{2} |x'|^2,$$
till either we touch $\p E$ or $\{x_n=\eps(1-\delta^2)\}.$ In the latter case, \eqref{1} holds. Otherwise, the first touching point $y_1$ may occur either on $\p_f E$ or on $\p_c E.$ Again, if the latter happens, then \eqref{2} is satisfied. The reason for this is that $y_1 \in \p \mathcal O$, and $\p E$ is above $\p \mathcal O$, and the obstacle satisfies the hypothesis \eqref{F3}.
 
If $y_1 \in \p_f E$, then by \eqref{KET200} we conclude that $f_E(y_1) \leq C \eps$ and therefore the other half of \eqref{fe} also holds. Now the result follows as in the unconstrained case.

\medskip

\textbf{Step 3 -- Compactness.} Fix $\rho>0$ to be specified later. Assume by contradiction that there exist a sequence $\eps_k \to 0$, $\delta_0 \to 0$, and a sequence $\{\p E_k\}$ 
satisfying \eqref{vsob1}-\eqref{vsob2}, $0\in \p E_k$, and a sequence $\varphi_k$ satisfying \eqref{F3}, such that 
$$
\p E_k \cap B_1 \subset \{|x_n| \leq \eps_k\},$$
$$ \int_{\R^n \setminus B_1} \frac{|\chi_{E} - \chi_{P_0}|}{|x|^{n+2s+1}} dx \leq \eps_k,
$$ but  $\p E_k$ does not satisfy the conclusion of the lemma. Then, by Steps 1and 2,  the sets
$$\p E^*_k := \{(x', \frac{x_n}{\eps_k}) \ | \ x \in \p E_k\}$$
converge uniformly on $B'_{3/4} \times [-1,1]$ (up to extracting a subsequence) to the set $$E^*_0:= \{(x', u(x'))\}$$
where $u: B'_{3/4} \to \R$ is H\"older continuous, $u(0)=0$ and $|u| \leq 1.$ 

Moreover,  $\varphi_k \to \varphi_0$ uniformly in $B_1'$ where $\varphi_0$ is a constant, possibly $-\infty$.

Finally, the functions 
 $$f_{E_k^*}(x):= \frac{1}{\eps_k}f_{E_k}(x', \eps_k x_n),$$
converge uniformly on the cylinder $B'_{3/4} \times [-1,1]$ to a limiting function $f_0$ which has bounded Lipschitz seminorm, and it is bounded below but not above i.e. $f_0$ could also be $+\infty$.

We show that $u$ satisfies the following obstacle problem, in the viscosity sense:
$$\Delta_T^{ \frac 12+s} u = \frac 12 f_0 \quad \text{in $B'_{1/2} \cap \{u>\varphi_0 \}$}, \quad \quad \quad \Delta_T^{ \frac 12+s} u \le \frac 12 f_0 \quad \text{in $B'_{1/2}$},   $$
and $\|f_0\|_{C^{0,1}} \le C$.

We only need to show that $f_0$ is bounded above, since then the claim follows as in Lemma \ref{flat_imp}.
We argue similarly as in Step 2.
First we notice that $u$ and $\varphi_0$ cannot coincide identically in $B'_{1/2}$ since then they 
would both vanish, and the conclusion of lemma would clearly hold for large $k$'s. 
This means that we can slide a parabola $x_n=C|x'-x_0'|^2$ with $x_0' \in B_{1/2}'$ by above and 
obtain a contact point $y'$ in $B_{5/8}'$ where $u>\varphi_0$. At this point we obtain $f_0(y') < C$ 
and our claim is proved. 

From the optimal regularity in the obstacle problem for fractional Laplacian (see \cite{CSS, CDS}) we find that $u \in C^{1,\frac 12 +s}(B_{1/4})$ and then we reach a contradiction as in Lemma \ref{flat_imp}.
\end{proof}

In the next lemma we estimate the difference between the fractional curvature $K_E$ and its linearization $\triangle ^{\frac 12 +s}$ for a $C^{1,\beta}$ graph with $\beta >2s$, in a neighborhood of a point which has horizontal tangent plane. Notice that since $\beta>2s$, the fractional curvature $K_E$ and $\triangle ^{\frac 12 +s}$ are bounded and well defined at all points of the graph.

\begin{lem}\label{change}Let $$\p E \cap \{|x'| \le 2\}:= \{(x', u(x')) \ | \ u \in C^{1,\beta}(B'_2)\}, \quad 1>\beta >2s,$$
$$\nabla u(0)=0, \quad u(0)=0, \quad \|u\|_{C^{1,\beta}(B'_2)} \leq 1.$$
Extend $u=0$ outside $B'_2.$ Then,
\be\label{700} 2\Delta^{\frac 1 2 +s} u(x') - K_E(x',u(x'))= g(x')\ee 
with $|g| \leq C$ and having the following modulus of continuity
\be\label{701} |g(x'_1)-g(x'_2)| \leq C (\max|x'_i|)^{2\beta} |x_1' - x_2'|^{\beta -2s} + C|x'_1-x'_2|, \ee 
with $C$ depending on $n,\beta, s.$
\end{lem}
\begin{proof}
Let $x'_1, x'_2 \in B'_r$ and $|x'_1-x'_2|=d$. Call $E_i $ the translation of $E$ by $-x_i=-(x'_i, u(x'_i))$ that is
$$E_i=E - x_i, \quad i=1,2.$$
Let $\p E_i = \{(y', v_i(y'))\},$ then
$$v_i(y')=u(x_i'+y')-u(x_i'),$$
and from the $C^{1,\beta}$ continuity of $u$ we obtain
\be \label{70} |v_1(y') - v_2(y')| \leq C |y'|^\beta  d.\ee
We have 
\begin{align*}
\Delta^{\frac 1 2 +s} u(x_i')&=\int_{B_d'} \frac{v_i(y')}{|y'|^{n+2s}}dy' + \int_{B_1'\setminus B_d} \frac{v_i(y')}{|y'|^{n+2s}}dy' + \int_{\C B'_1(x_i')} \frac{u(x')-u(x_i')}{|x'-x_i'|^{n+2s}}dx' \\
&=: I_1(x_i) + I_2(x_i)+I_3(x_i)
\end{align*}
and
\begin{align*}
K_E(x_i',u(x_i'))=&\int_{B_d' \times [-1,1]} \frac{\chi_{E_i}- \chi_{\C E_i}}{|y|^{n+2s}}dy + \int_{(B_1'\setminus B_d') \times [-1,1]}  \frac{\chi_{E_i} - \chi_{\C E_i}}{|y|^{n+2s}}dy \\
& + \int_{\C (B'_1(x_i') \times [-1,1])} \frac{\chi_E-\chi_{\C E}}{|x-x_i|^{n+2s}}dx \\
=:& J_1(x_i) + J_2(x_i)+J_3(x_i).
\end{align*}
First we show that
\be\label{71}
J_1(x_i)=2I_1(x_i)+ O \left (r^{2 \beta} d ^{\beta-2s} \right).\ee
Indeed, $J_1(x_i)$ vanishes if we replace $E_i$ by the subgraph $T_i$ of the tangent plane of $\p E_i$ at 0. Hence
$$J_1(x_i)=2 \int_{B'_d\times [-1,1]} \frac{\chi_{E_i} - \chi_{T_i}}{|y|^{n+2s}} dy.$$
In the region where $\chi_{E_i} \ne \chi_{T_i}$ we use that the tangent plane has slope at most $C r^\beta$ and we write
$$\frac{1}{|y|^{n+2s}} = \frac{1}{|y'|^{n+2s}} (1+O(r^{2\beta})).$$
 We obtain
$$J_1(x_i) = 2 \int_{B'_d} \frac{v_i(y') - \nabla v_i(0) \cdot y'}{|y'|^{n+2s}} dy' + O(r^{2\beta }d^{\beta-2s}),$$
since
$$\int_0^d t^{1+\beta} t^{-n-2s} r^{2\beta} t^{n-2} dt = C r^{2\beta} d^{\beta-2s},$$
and we proved \eqref{71}.

Next we show that 
\be \label{72}
J_2(x_2)-J_2(x_1)=2(I_2(x_2)-I_2(x_1)) + O \left (r^{2 \beta} d ^{\beta-2s} +d\right).\ee
We have
$$J_2(x_2)-J_2(x_1)= 2 \int_{(B'_1 \setminus B'_d) \times [-1,1]} \frac{ \chi_{E_2} - \chi_{E_1}}{|y|^{n+2s}} dy,$$
and in the set where $\chi_{E_1}(y) \ne \chi_{E_2}(y)$ we have
$$\frac{1}{|y|^{n+2s}} = \frac{1}{|y'|^{n+2s}} (1+O\left ( (r+|y'|)  ^{2\beta} \right)).$$
Thus
$$J_2(x_2)-J_2(x_1)=2 \int_{B_1'\setminus B_d'} \frac{v_2(y')-v_1(y')}{|y'|^{n+2s}} dy' + O(\gamma)$$
and by \eqref{70} we have
$$\gamma:= d \int_d^1  t^\beta t^{-n-2s}(r+t)^{2\beta} t^{n-2} dt \le C( r^{2 \beta} d^{\beta-2s} + d)$$
and we proved \eqref{72}.

Finally, it easy to check that 
$$ |J_3(x_2)-J_3(x_1)| \le Cd,\quad \quad |I_3(x_2)-I_3(x_1)| \le Cd, $$
since the domain of integrations for $J_3(x_1)$ and $J_3(x_2)$, and $I_3(x_1)$, $I_3(x_2)$ respectively, differ by a set of measure proportional to $d$, and in the common domain of integration we use
$$\frac{1}{|x-x_2|^{n+2s}}-\frac{1}{|x-x_1|^{n+2s}}=O(\frac{d} {|x|^{n+1+2s}}),$$
respectively
$$ \frac{1}{|x'-x'_2|^{n+2s}}-\frac{1}{|x'-x'_1|^{n+2s}}=O(\frac{d} {|x'|^{n+1+2s}}), \quad \quad  u(x'_2) -u(x_1')=O(d).$$
Now the conclusion follows from the inequalities above and \eqref{71}, \eqref{72}.

\end{proof}

As a consequence of Lemma \ref{change} and the Almgren monotonicity formula in \cite{CDS} we obtain the optimal regularity in the obstacle problem for the nonlocal minimal 
surfaces in the case when the 
constrained minimal surface $\p E$ is of class $C^{1,\beta}$ with $\beta$ close to the optimal exponent $\frac 12 +s$. 

\begin{cor}\label{c6} Let $\p E=\{(x',u(x')| \quad x' \in B_2'\}$ be a nonlocal minimal surface constrained 
above an obstacle $\p \mathcal O=\{(x',\varphi(x')  \}$ as in
Theorem \ref{main_reg_OMS}. Assume that $u \in C^{1,\beta}$ for $\beta =\frac 12 + s -\frac{\delta}{2}$, 
and $\varphi \in C^{1,\alpha}$ with $\alpha > \frac 12 + s$. 
Assume further that $0 \in \p_{\R^{n-1}} \{u>\varphi\}$ and $\nabla u(0)=0$. 
Then $u$ is pointwise $C^{1,\frac 12 + s}$ at the origin.
\end{cor}

Indeed, $u$ solves the obstacle problem for $\Delta^{\frac 12 +s}$ with obstacle $\varphi$ and 
right hand side $\frac 12 g(x')$ with $g$ as in \eqref{700}. Moreover, from the modulus of continuity of $g$ given in \eqref{701}, and since 
$3\beta -2s>1$, we have the following H\"older bounds for $g$ in balls $B_r'$ centered at the origin:
$$[g]_{C^\delta(B_r')} \le C r^{1-\delta} \le C r^{1-(\frac 12 +s)}.$$
Now Proposition 6.4 in \cite{CDS} applies and we obtain that $u$ is pointwise $C^{1,\frac 12 +s}$ at the origin.
For convenience of the reader, we state this proposition below (in our case $\bar s =1/2+s$).

\begin{prop} \label{ls5} Let $u \in C^{2\bar s+\eps}, \bar s \in [\frac 12, 1),$ solve the obstacle problem\begin{equation}\label{opti3}
u \ge \varphi \quad \mbox{in $B_1$} ,
\end{equation}
\begin{equation}\label{opti4}
\mathcal \triangle^{\bar s} u \le g \quad \mbox{in $B_1$}, \quad \mbox{and} \quad \quad \quad \mathcal \triangle^{\bar s} u = g \quad \mbox{in} \quad \{u>\varphi\} \cap B_1.  
\end{equation}
Assume that $\|u\|_{L^1(\R^n,d\omega )} \le 1$ and $\nabla u$ is pointwise $C^{\bar s- \frac \delta 2}$ at $0 \in \p \{u=\varphi\}$, i.e.
\begin{equation}\label{uo}
|\nabla u(x)| \le |x|^{\bar s-\frac \delta 2}\quad \quad \mbox{in $B_1$.}
\end{equation}
If $\varphi \in C^{2\bar s+\eps}$, $\nabla \varphi$ is pointwise $C^{\bar s+\delta}$ at the origin i.e., for all $r <1$
$$ [\nabla \varphi]_{C^{2\bar s + \delta -1}(B_r)} \le r^{1-\bar s},$$
and $g$ satisfies
$$[g]_{C^\delta(B_r)} \le C r ^{1-\bar s}$$ 
then $u$ is pointwise $C^{1,\bar s}$ at the origin i.e. 
\be\label{conu}
|u(x)| \le C |x|^{1+\bar s} \quad \quad \mbox{in $B_{1}$},\ee
for some $C$ depending only on $n$, $\bar s$ and $\delta$.
\end{prop}

\smallskip

We are now ready to provide the proof of our main Theorem \ref{main_reg_OMS}.

\medskip

\noindent {\it Proof of Theorem $\ref{main_reg_OMS}$.} In view of Lemma \ref{f_red},
we can apply Theorem \ref{imp_OMS} for a dilation of $E$ and obtain that locally 
$\p E=\{(x',u(x'))\}$ is a $C^{1,\beta}$ graph with $\beta$ sufficiently close to $\frac 12 + s$.
Then we apply Corollary \ref{c6} and obtain that $u$ is pointwise $C^{\frac 12 +s}$ at all points on 
the free boundary $\p_{R^{n-1}}\{u> \varphi \} \cap B'_{\delta_0}$. Now it is standard to extend this regularity of $u$ 
in a full neighborhood of the origin, in view of Theorem \ref{imp_MS}.

\qed

\section{The two membranes problem}

In this section, we use the result obtained in Section 2, to prove our main Theorem \ref{T4}. 

Let $\mathcal G$ be the functional defined in the Introduction,
$$
\mathcal G(E,F) : =\mathcal P^s_{B_1}(E) + Per_{B_1}(F) + \int_{B_1} (f \chi_E  + g \chi_F) dx $$
with $f, g \in L^\infty(B_1).$
Let $E_0$, $F_0$ be two sets such that $F_0 \cap B_1 \subset E_0 \cap B_1$ and $\mathcal G
(E_0,F_0) < \infty$.  

We minimize $\mathcal G$ in the class of admissible pairs of sets in $\R^n$:
$$\mathcal A:= \{(E,F) \ | F \subset E \quad \mbox{in $B_1$}, \quad F=F_0, \quad E=E_0 \quad 
\mbox{outside $B_1$}  \}.$$

 By the compactness of the spaces 
$H^s(B_1)$, $BV(B_1)$ in $L^1(B_1)$ 
and the lower semicontinuity of $\mathcal G$ with respect to the $L^1$ convergence, 
we obtain the existence of a minimizing pair $(E,F)$.  We study here the regularity properties of this pair.

For simplicity we take $f=g=0$ since the general case follows similarly. We need the following preliminary results.

\begin{lem}\label{almost}$\p F$ is an almost minimal surface i.e. for any compact perturbation $F_0$ of $F$ in any ball $B_r(x_0) \subset B_1$ we have
$$Per_{B_2}(F) \le Per_{B_2}(F_0) + C r^{n-1+2\sigma}, \quad 2\sigma=1-2s.$$ In particular, $\p F$ is a $C^{1,\sigma}$ surface, except possibly on a singular set $\Sigma$ of Hausdorff dimension $n-8$. \end{lem}
\begin{proof}
Indeed, the pair $(E_0,F_0)$ with $E_0=E \cup B_r(x_0)$ is an admissible pair since $F_0 \cap B_1 \subset E_0 \cap B_1$. Thus, the minimality of $(E,F)$ gives
\begin{align*} Per_{B_2}(F) - Per_{B_2}(F_0) & \le \mathcal P^s_{B_1}(E_0) -  \mathcal P^s_{B_1}(E) \\
& = L(E_0 \setminus E, \C E \setminus B_r(x_0)) - L(E_0 \setminus E, E) \\
& \le L(B_r(x_0), \C B_r(x_0)) \\
& \le C r^{n-2s}, 
\end{align*}
where $L(A,B)$ is defined in \eqref{L}.

In conclusion $\p F$ is an almost minimal surface in the sense of Almgren and Tamanini \cite{A,T}, which means that $\p F$ is a $C^{1,\sigma}$ surface, with $2\sigma=1-2s$, except possibly on a singular set of Hausdorff dimension $n-8$.

\end{proof}

\begin{lem}\label{visc_F} Assume $F$ is a subgraph of a $C^{1,\sigma}$ graph in the $e_n$ direction in $B_1$.
Let $\kappa_F$ and $K_E$ denote respectively the mean curvature of $\p F$ and the $s$-fractional curvature of $\p E,$ and set $Q=\p E \cap \p F$. Then
\begin{enumerate}
\item the following holds in the viscosity sense:\begin{align}\label{11}&\kappa_F \ge 0, \quad K_E \le 0, \\
\label{12} &\kappa_F=0, \quad K_E=0 \quad \mbox{away from $Q$} ,\\
\label{13}&\kappa_F + 2 K_E\leq 0 \quad \mbox{on $Q$.}\end{align}
\item If $\p E \in C^{1,\beta}$ with $\beta >2s$, the following holds in the viscosity sense:
\be \label{14}\kappa_F +  K_E= 0 \quad \mbox{on $Q$.}\
\ee
\end{enumerate}
\end{lem}
\begin{proof} For part (i), we only prove \eqref{13}, as \eqref{11}-\eqref{12} are standard (see also \eqref{vsob1}-\eqref{vsob2}.) The fact that this equation is satisfied in a viscosity sense means that if $\Gamma$ is a $C^2$ surface which touches $\p F$ by below at $x_0 \in Q$ then the $s$-fractional curvature $K_E$ is well defined at $x_0$ and we request that 
$$\kappa_\Gamma(x_0) + 2 K_E(x_0) \leq 0.$$ 
Assume $0 \in Q$ and that $\Gamma \in C^2$ touches $\p F$ strictly by below at $0$, say $\Gamma$ is given by $x_n=p$ with $p$ a quadratic polynomial. Call $P_\eps$ the subgraph of $x_n=p+\eps$ (with $\eps$ small), and set $$A_\eps := P_\eps \cap \mathcal C F, \quad D_\eps:= P_\eps \cap \mathcal C E.$$

A standard computation, obtained by integrating by parts $\Delta d$ over $A_\eps$, with $d$ the signed distance function from $x$ to $\Gamma$ (positive above $\Gamma$ in the $e_n$ direction), gives that 
$$\int_{A_\eps} div(\nabla d) dx = \int_{\p A_\eps}  \nabla d \cdot \nu d \mathcal H^{n-1}$$
hence
\be\label{standard_p} |A_\eps| (-\kappa_\Gamma(0)+o_\eps(1)) \geq Per_{B_1}(F \cup A_\eps) -Per_{B_1}(F).
\ee

Unfortunately from the existing literature it is not clear if one can use the same perturbation set $D_\eps$ to obtain a similar inequality for the fractional curvature. Indeed in Theorem 5.1 \cite{CRS} the authors use the perturbation of $E$ by $D_\eps \cup T(D_\eps)$, where $T$ denotes the reflection with respect to $\Sigma:= \Gamma+\eps e_n.$ In this way they can use the symmetry of the energy functional, to obtain the Euler-Lagrange equation. Precisely, 
call $$D^*_\eps:= D_\eps \cup T(D_\eps).$$ Then according to Theorem 5.1 in \cite{CRS} we have that for a sequence of $\eps \to 0$
\be\label{fract_p}
|D^*_\eps| (-K_E(0)+o_\eps(1)) \geq \mathcal P_{B_1}^s(E\cup D^*_\eps) - \mathcal P_{B_1}^s(E).
\ee
Notice that $(E \cup D^*_\eps, F \cup A_\eps)$ is an admissible perturbation. Thus, adding up \eqref{standard_p}-\eqref{fract_p} and using the minimality of $(E,F)$ we get that
$$- (\kappa_\Gamma(0)+o_\eps(1)) + \frac{|D^*_\eps|}{|A_\eps|}(-K_E(0)+ o_\eps(1)) \geq 0.$$
Using that $K_E(0) \leq 0$ and $|D^*_\eps| = (2+ o_\eps(1))|D_\eps|$ (and $D_\eps \subset A_\eps$) we obtain the desired claim as $\eps \to 0$.

To prove the claim in (ii) it suffices to repeat the argument above and show that in the case when $\p E$ is sufficiently smooth, then \eqref{fract_p} holds with $D^*_\eps$ replaced by $D_\eps.$ This would lead to $$\kappa_\Gamma(0) + K_E(0) \leq 0.$$
The other inequality can be proved similarly.

We want to show that
$$ \liminf_{\eps \to 0} \frac{1}{|D_\eps|} \left( L(D_\eps, \mathcal C E \setminus D_\eps) - L(D_\eps, E) \right) \le -K_E(0).$$
The double integrals $L(D_\eps, \cdot)$ above are of order greater than $|D_\eps|$ with most of the contribution coming when both $x$, $y$ are sufficiently close to the origin.
In order to obtain the inequality above it suffices to show as in [CRS] that these contributions cancel each others out near the origin i.e.    
\be \frac{1}{|D_\eps|} \left( L_\delta(D_\eps, \mathcal C E \setminus D_\eps) - L_\delta(D_\eps, E) \right) \leq o_\delta(1)\ee
with $o_\delta(1) \to 0$ as $\delta \to 0,$ and for $\eps \ll \delta$, where
\be\label{Ldelta}
L_\delta(A,B): = \int_{\{|x-y|<\delta\}} \frac{1}{|x-y|^{n+2s}} \chi_A(x)\chi_B(y) dx dy.\ee

 We claim that if $A$ is a set such that $\partial A \cap B_1$ is a surface with bounded $C^{1,2s+\sigma}$ norm, and $x \notin  A$ with the distance  $d_A(x)$ from $x $ to $A$ satisfying $d_A(x) \ll \delta$ then
 \be\label{da}
 \int _{B_\delta(x)}  \frac{1}{|x-y|^{n+2s}} \chi_A(y) dy = h(d_A) + o_\delta(1),
 \ee
where $h(d)$ is a decreasing function comparable to $d^{-2s}$.

Indeed, assume for simplicity that $x=0$ and the distance from $x$ to $A$ is realized at $d \, e_n$, and let $P_d=\{x_n >d \}$. Since in $B_1$
$$\{x_n >d+ |x'|^{1+2s+\sigma}\} \subset A \subset \{x_n >  d-|x'|^{1+2s+\sigma}\},$$
we find that $$r^{-(n+2s)}\int_{B_r}|\chi_A -\chi_{P_d}| dy \le C r^ \sigma,$$
hence 
 $$\int _{B_\delta}  \frac{1}{|y|^{n+2s}} \chi_A(y) dy \le  \int _{B_\delta}  \frac{1}{|y|^{n+2s}} \chi_{P_d}(y) dy + C \sum_{d \le r=2^{-k} \le \delta} r^\sigma :=h(d) + o_\delta(1).$$
Clearly $h(d)$ is decreasing with $d$ and after rescaling we see that
$$h(d)=d^{-2s} \int _{B_{\delta/d}}  \frac{1}{|y|^{n+2s}} \chi_{P_1}(y) dy \quad \sim \quad d^{-2s}.$$

This proves claim \eqref{da} and we conclude that if $\p A$ is a $C^{1,2s+\sigma}$ surface and $D$ is a measurable set outside $A$ then
$$L_\delta (D,A)= \int_D  h(d_{A}) dx + o_\delta(1).$$

Let $x_\Sigma$ denote the projection of $x$ on the $C^2$ surface $$\Sigma:=\Gamma + \eps e_n.$$ The set $D_\eps$ is the set between the graphs given by $\Sigma$ and $\p E$ and can written as the region between $\Sigma$ and a graph $\gamma$ with respect to $\Sigma$ i.e.  
$$D_\eps = \{0 < d_\Sigma(x) < \gamma (x_\Sigma)  \} ,$$ 
with $\gamma \in C^{1,2s+\delta}$ a function supported on an $(n-1)$-dimensional domain $ S \subset \Sigma$, and also $\gamma \to 0$ as $\eps \to 0$.

We have
$$L_\delta(D_\eps, \mathcal C E \setminus D_\eps) \le L_\delta(D_\eps, \{x_n \ge P_d+ \eps  \}),$$
hence
$$ L_\delta(D_\eps, \mathcal C E \setminus D_\eps) - L_\delta(D_\eps, E) \le \int_{B_\eps}  h(d_{\Sigma})-h(d_{E}) dx +o_\delta(1).$$
Since $$d_{E}(x) \le f(x_\Sigma) - d_\Sigma(x)$$
and we change variables
$$x=x_\Sigma + \nu d_{\Sigma}(x)= z + \nu_z t, \quad \quad \mbox{with} \quad z \in \Sigma, \quad t \in \R$$
hence
$$dx=G(z,t) dz dt,$$
for a Lipschitz function $G$ defined on $\Sigma \times \R$.
We obtain
$$\int_{B_\eps}  h(d_{\Sigma})-h(d_{E}) dx \le \int_S \int_0^{f(d_\Sigma(z))} \left(h(t)-h(f(d_\Sigma(z))-t)\right) G(z,t) dt dz.$$
Since $G$ is Lipschitz we have $G(z,t)=G(z,0) + O(t)$ hence for fixed $z$ we have
\begin{align*}&\int_0^{f(d)} \left(h(t)-h(f(d)-t)\right) G(z,t) dt\\ &\le G(z,0) \int_0^{f(d)} \left(h(t)-h(f(d)-t)\right)  dt + C \int_0^{f(d)} h(t) t  \, dt\\
& \le C \int_0^{f(d)} t^{1-2s} \, dt \le C f(d) ^{2-2s} = o_\eps(1) f(d).\end{align*}
In conclusion
$$\int_{D_\eps}  h(d_{\Sigma})-h(d_{E}) dx \le o_\eps(1) \int_S f(d_\Sigma(z)) dz \le o_\eps(1) |D_\eps| ,$$
and the proof is completed.
\end{proof}

We are now ready to prove our main Theorem \ref{T4}.

\medskip

{\it Proof of Theorem $\ref{T4}$.} 
From Lemma \ref{almost}, we know that $\p F$ is a $C^{1,\sigma}$ surface, except possibly on a singular set $\Sigma$ of Hausdorff dimension $n-8$. 

We divide the proof in two steps.

\smallskip

{\bf Step 1.} Assume that $0 \in \p F \setminus \Sigma$. We claim that $\p F$ is a $C^{1,\beta}$ surface in a neighborhood of the origin for any $\beta <1$. 

Indeed, near $0$, $F$ is a subgraph of a $C^{1,\sigma}$ graph in some direction, say the $e_n$ direction. Let $\kappa_F$ and $K_E$ denote the mean curvature of the graph $\p F$ and the $s$-fractional curvature of $\p E$ respectively, and let $Q:=\p F \cap \p E$.

Since $F \subset E$ in $B_1$ we have that $K_F$ i.e. the $s$-fractional curvature of $\p F$, satisfies $ K_F \le K_E + C$ at the points on $\p F \cap \p E$ near the origin. We deduce from $(i)$ in Lemma \ref{visc_F} that $\p F$ satisfies the following inequalities in the viscosity sense in a neighborhood of the origin.
$$\kappa_F \ge 0, \quad \kappa_F + 2\chi_{Q} K_F \le C.$$
We claim that the inequalities above give $\p F \in C^{1,\beta}$ for any $\beta<1$. We sketch some of the arguments (see also Proposition 4.11 in \cite{CDS}). 

The rescaling $\tilde F=\frac 1r F$ satsfies
$$\kappa_{\tilde F} \ge 0, \quad \kappa_{\tilde F} +2 r^{1-2s} \chi_{ \tilde Q} K_{\tilde F} \le C r,$$
hence after an initial rescaling we may assume that $\p F$ is sufficiently flat in $B_1$ and satisfies
\be\label{s81}
\kappa_{F} \ge 0, \quad \kappa_{ F} + a \chi_{Q} K_ F \le 1, \quad \mbox{for some $a\in [0,2]$.}\ee

Then we show by induction that there exists a sequence of unit vectors $e_k$ such that
$$\p F \cap B_r \subset \{|x \cdot e_k| \le r^{1+\beta}\},$$
for $r=\rho^k$, provided that the inclusion above already holds for $k \le k_0$. Here $\rho$, $k_0$ are universal constant that depends only on $n$, $s$ and $\beta$ (but not on $a$).

The existence of $k_0$ follows by compactness. Let $K_{r,F}^T$ be the truncated curvature of distance $r$ i.e.
$$K_{r,F}^T(x):=\int_{B_r(x)} \frac{\chi_F-\chi_{\C F}}{|y-x|^{n+2s}}dy.$$
We replace $K$ by $K_{ r/2}^T$ in \eqref{s81}. Assume that $\beta>2s$ and use the induction hypothesis to obtain 
$$ |K_F-K^{T}_{\frac r2,F}| \le C \quad \mbox{on $\p F \cap B_r$}.$$
Thus the rescaling $\tilde F=\frac 1r F$ satisfies in $B_1$
$$\kappa_{\tilde F} \ge 0, \quad \kappa_{\tilde F} + a \,  r^{1-2s}  \, \chi_{ \tilde Q} \, K^{T}_{\frac 12,\tilde F} \le  Cr,$$
and $\p \tilde F \subset \{|x_n| \le r^\beta \} \cap B_1$. 

We are now in the position to apply the ``$\eps$-flatness implies regularity" theory developed in \cite{Sa}, ($\eps=r^\beta$).
The reason for this is that if $P$ is a paraboloid of size $\eta$ that touches $\p \tilde F$ by below at $y$, then $K^T_{\frac 12,\tilde F}(y)\geq -c \eta.$ Thus the inequalities above are perturbations of order smaller than $r^\beta$of the equation $\kappa_{\tilde F}=0$. Hence Harnack inequality holds and it guarantees compactness in the limit (see Theorem 1.1 in \cite{Sa}). This means that as $r \to 0$, the stretching of factor $r^{-\beta}$ in the $x_n$ direction of the sets $\p \tilde F$ converge uniformly to the graph of a harmonic function, and therefore we obtain the desired conclusion.

\smallskip

{\bf Step 2.} Let $0 \in \p E \cap \p F \setminus \Sigma$. By step 1, $\p F$ is $C^{1,\beta}$ for any $\beta<1$ in a neighborhood of the origin, thus we can apply Theorem \ref{main_reg_OMS} and obtain that near $0$, $\p E$ is a graph of a $C^{1,\frac 12 +s}$ function $u$. This means that $K_E(x',u(x'))$ is a $C^{\frac 12 -s}$ function of $x'$ (see Lemma \ref{change} for example). From the Euler-Lagrange equation in $(ii)$ Lemma \ref{visc_F}, we obtain $\kappa_F$ is a $C^{\frac 12 -s}$ function of $x'$, hence $\p F \in C^{2,\frac 12 -s}$.

\qed

\end{document}